\documentclass[a4paper]{amsart} 
\usepackage{amssymb} 
\usepackage{verbatim}
\usepackage[all]{xy}
\SelectTips{cm}{}

\usepackage{hyperref}
\calclayout

\title{Cohomological  length functions}

\thanks{Version from May 21, 2013.}

\author{Henning Krause}
\address{Fakult\"at f\"ur Mathematik\\
Universit\"at Bielefeld\\ D-33501 Bielefeld\\ Germany}
\email{hkrause@math.uni-bielefeld.de}

\newtheorem{lem}{Lemma}[section]
\newtheorem{prop}[lem]{Proposition}
\newtheorem{cor}[lem]{Corollary}
\newtheorem{thm}[lem]{Theorem}
\newtheorem{conj}[lem]{Conjecture}

\theoremstyle{remark}
\newtheorem{rem}[lem]{Remark}

\theoremstyle{definition}
\newtheorem{exm}[lem]{Example}
\newtheorem{defn}[lem]{Definition}

\numberwithin{equation}{section}

\newcommand{\smatrix}[1]{\left[\begin{smallmatrix}#1\end{smallmatrix}\right]}

\renewcommand{\mod}{\operatorname{\mathsf{mod}}\nolimits}

\newcommand{\Ob}{\operatorname{\mathsf{Ob}}\nolimits}
\renewcommand{\sup}{\operatorname{\mathsf{sup}}\nolimits}
\newcommand{\Sp}{\operatorname{\mathsf{Sp}}\nolimits}

\newcommand{\proj}{\operatorname{\mathsf{proj}}\nolimits}
\newcommand{\per}{\operatorname{\mathsf{per}}\nolimits}

\newcommand{\rad}{\operatorname{\mathsf{rad}}\nolimits}

\newcommand{\id}{\operatorname{id}\nolimits}

\newcommand{\Mod}{\operatorname{\mathsf{Mod}}\nolimits}

\newcommand{\End}{\operatorname{\mathsf{End}}\nolimits}
\newcommand{\Hom}{\operatorname{\mathsf{Hom}}\nolimits}
\newcommand{\RHom}{\operatorname{\mathsf{RHom}}\nolimits}

\renewcommand{\Im}{\operatorname{\mathsf{Im}}\nolimits}
\newcommand{\Ker}{\operatorname{\mathsf{Ker}}\nolimits}
\newcommand{\KGdim}{\operatorname{\mathsf{KGdim}}\nolimits}

\newcommand{\coh}{\operatorname{\mathsf{coh}}\nolimits}

\newcommand{\Ext}{\operatorname{\mathsf{Ext}}\nolimits}

\newcommand{\Lex}{\operatorname{\mathsf{Lex}}\nolimits}

\newcommand{\Spec}{\operatorname{\mathsf{Spec}}\nolimits}
\newcommand{\length}{\operatorname{\mathsf{length}}\nolimits}

\newcommand{\Thick}{\operatorname{\mathsf{Thick}}\nolimits}
\newcommand{\rep}{\operatorname{\mathsf{rep}}\nolimits}

\newcommand{\Zsp}{\operatorname{\mathsf{Zsp}}\nolimits}

\newcommand{\Ab}{\operatorname{\mathsf{Ab}}\nolimits}
\newcommand{\op}{\mathrm{op}}

\newcommand{\comp}{\mathop{\circ}}
\newcommand{\lto}{\longrightarrow}
\newcommand{\xto}{\xrightarrow}

\def\a{\alpha}
\def\b{\beta}
\def\e{\varepsilon}

\def\g{\gamma}
\def\p{\phi}

\def\Ga{\Gamma}

\def\Si{\Sigma}

\def\A{{\mathsf A}}
\def\B{{\mathsf B}}
\def\C{{\mathsf C}}
\def\D{{\mathsf D}}

\def\Sc{{\mathsf S}}

\def\T{{\mathsf T}}
\def\U{{\mathsf U}}
\def\V{{\mathsf V}}

\def\bbN{{\mathbb N}}
\def\bbP{{\mathbb P}}

\def\bbZ{{\mathbb Z}}

\def\frp{{\mathfrak p}}
\def\frq{{\mathfrak q}}

\begin{document}

\begin{abstract}
  We study certain integer valued length functions on triangulated
  categories and establish a correspondence between such functions and
  cohomological functors taking values in the category of finite
  length modules over some ring. The irreducible cohomological
  functions form a topological space. We discuss its basic properties
  and include explicit calculations for the category of perfect
  complexes over some specific rings.
\end{abstract}


\maketitle

\section{Introduction}

Let $\C$ be a triangulated category with suspension
$\Si\colon\C\xto{\sim}\C$. Given a cohomological functor
$H\colon\C^\op\to \Mod k$ into the category of modules over some ring
$k$ such that $H(C)$ has finite length for each object $C$, we
consider the function
\[\chi\colon \Ob\C\lto\bbN,\quad C\mapsto \length_k H(C),\]
and ask:
\begin{itemize}
\item[--] what are the characteristic properties of such a function
  $\Ob\C\lto\bbN$, and
\item[--] can we recover $H$ from $\chi$?
\end{itemize}
Somewhat surprisingly,  we can offer fairly complete answers to
both questions.

Note that similar questions arise in Boij--S\"oderberg theory when
cohomology tables are studied; recent progress \cite{EE2012,ES2009}
provided some motivation for our work. Further motivation comes
from the quest (initiated by Paul Balmer, for instance) for points in
the context of triangulated categories.

Typical examples of cohomological functors are the representable
functors of the form $\Hom(-,X)$ for some object $X$ in $\C$. We begin
with a result that takes care of this case; its proof is elementary and based on a
theorem of Bongartz \cite{Bo1989}.

\begin{thm}[Jensen--Su--Zimmerman \cite{JSZ2005}]\label{th:1}
  Let $k$ be a commutative ring and $\C$ a $k$-linear triangulated
  category such that each morphism set in $\C$ has finite length as a
  $k$-module. Suppose also for each pair of objects $X,Y$ that
  $\Hom(X,\Si^nY)=0$ for some $n\in\bbZ$. Then two objects $X$ and $Y$
  are isomorphic if and only if the lengths of $\Hom(C,X)$ and
  $\Hom(C,Y)$ coincide for all $C$ in $\C$.\qed
\end{thm}

Examples of triangulated categories satisfying the assumptions in this
theorem arise from bounded derived categories. To be precise, if $\A$
is a $k$-linear exact category such that each extension group in $\A$
has finite length as a $k$-module, then its bounded derived category
$\D^b(\A)$ satisfies the above assumptions, since for all objects
$X,Y$ in $\A$ (viewed as complexes concentrated in degree
zero)
\[\Hom(X,\Si^nY)=\Ext^{n}(X,Y)= 0 \quad\text{for all
}n<0.\]

On the other hand, Auslander and Reiten provided in
\cite[Section~4.4]{AR1986} simple examples of triangulated categories
which do not have the property that objects are determined by the
lengths of their morphism spaces; see also \cite{BP1984}.

Now let $\C$ be an essentially small triangulated category. Thus the
isomorphism classes of objects in $\C$ form a set. Given any additive
functor $H\colon\C^\op\to\Ab$ into the category of abelian groups, we
denote by $\End(H)$ the endomorphism ring formed by all natural
transformations $H\to H$. Note that $\End(H)$ acts on $H(C)$ for all
objects $C$. Moreover, if a ring $k$ acts on $H(C)$ for all objects
$C$ in a way that commutes with all morphisms in $\C$, then this
action factors through that of $\End(H)$ via a homomorphism $k\to\End
(H)$. In particular, when $H(C)$ has finite length over $k$ then it
has also finite length over $\End(H)$.  This observation motivates the
following definition.

\begin{defn} 
  A cohomological functor $H\colon\C^\op\to\Ab$ is called
  \emph{endofinite}\footnote{The term \emph{endofinite} reflects
    condition (1), while (2) is added for technical reasons.} provided
  that for each object $C$ 
\begin{enumerate}
\item $H(C)$ has finite length as a module over the ring $\End(H)$, and
\item $H(\Si^n C)=0$ for some $n\in\bbZ$.
\end{enumerate} 
If $(H_i)_{i\in I}$ are cohomological functors, then the \emph{direct
  sum} $\bigoplus_{i\in I}H_i$ is cohomological.  A non-zero
cohomological functor $H\colon\C^\op\to\Ab$ is \emph{indecomposable}
if it cannot be written as a direct sum of two non-zero cohomological
functors.
\end{defn}

An endofinite cohomological functor $H\colon\C^\op\to\Ab$ gives rise
to a function
\[\chi_H\colon\Ob\C\lto\bbN,\quad C\mapsto\length_{\End(H)}H(C)\]
which is cohomological in the following sense; see
Lemma~\ref{le:fct}.

\begin{defn}\label{de:fct}
  A function $\chi\colon\Ob\C\to\bbN$ is called \emph{cohomological}
  provided that \begin{enumerate}
\item $\chi(C\oplus C')=\chi(C)+\chi(C')$ for each pair of objects $C$
  and $C'$,
\item for each object $C$ there is some $n\in\bbZ$ such that $\chi(\Si^n C)=0$,  and 
\item for each exact triangle $A\to B\to C\to $ in $\C$ and each
  labelling  \[\cdots \to X_{-2}\to X_{-1}\to X_0\to X_1 \to X_{2}\to \cdots\] of the induced
  sequence
\[\cdots\to\Si^{-1}B\to\Si^{-1}C\to A\to B\to C\to\Si A\to\Si B\to
\cdots\] with $\chi(X_0)=0$ we
have \[\sum_{i=0}^n(-1)^{i+n}\chi(X_i)\ge 0\qquad \text{for all
  $n\in\bbZ$},\] and equality holds when $\chi(X_n)=0$.
\end{enumerate} 
If $(\chi_i)_{i\in I}$ are cohomological functions and for any $C$ in
$\C$ the set $\{i\in I\mid \chi_i(C)\neq 0\}$ is finite, then we can
define the \emph{locally finite sum} $\sum_{i\in I}\chi_i$.  A
non-zero cohomological function is \emph{irreducible} if it cannot be
written as a sum of two non-zero cohomological functions.
\end{defn}

The following theorem is the main result of this note; it is proved in
\S\ref{se:main_thm} and builds on work of Crawley-Boevey on finite
endolength objects \cite{CB1994a,CB1994b}.

\begin{thm}\label{th:2}
Let $\C$ be an essentially small triangulated category.
\begin{enumerate}
\item Every endofinite cohomological functor $\C^\op\to\Ab$ decomposes
  essentially uniquely into a direct sum of indecomposable endofinite
  cohomological functors with local endomorphism rings.
\item Every cohomological function $\Ob\C\to\bbN$ can be written
  uniquely as a locally finite sum of irreducible cohomological
  functions.
\item The assignment $H\mapsto \chi_H$ induces a bijection between the
  isomorphism classes of indecomposable endofinite cohomological
  functors $\C^\op\to\Ab$ and the irreducible cohomological functions
  $\Ob\C\to\bbN$.
\end{enumerate}
\end{thm}

Examples of endofinite cohomological functors arise from representable
functors of the form $\Hom(-,X)$ when $\C$ is a Hom-finite $k$-linear
category. Thus Theorem~\ref{th:1} can be deduced from
Theorem~\ref{th:2}. The following remark shows that in some
appropriate setting each endofinite cohomological functor is
representable.

\begin{rem}\label{re:representable}
  Let $\T$ be a compactly generated triangulated category and $\C$ be
  the full subcategory formed by all compact objects. Then each
  endofinite cohomological functor $H\colon\C^\op\to\Ab$ is isomorphic
  to $\Hom(-,X)|_\C$ for some object $X$ in $\T$, which is unique up
  to an isomorphism and represents the functor\footnote{The functor is
    cohomological since $H$ is an injective object in the category of
    additive functors $\C^\op\to\Ab$; see the proof of
    Theorem~\ref{th:2}. Thus Brown's representabilty theorem applies.}
  \[\T^\op\lto\Ab, \quad C\mapsto\Hom(\Hom(-,C)|_\C,H).\] Thus
  cohomological functions are ``represented'' by objects in this
  setting. For specific examples, see \cite{KR2001}.
\end{rem}

Next we consider the set of irreducible cohomological functions
$\Ob\C\to\bbN$ and endow it with the Ziegler topology; see
Proposition~\ref{pr:top}.  The quotient
\[\Sp\C=\{\chi\colon \Ob\C\to\bbN\mid \text{$\chi$
  irreducible and cohomological}\}/\Si\] with respect to the action of
the suspension is by definition the \emph{space of cohomological
  functions} on $\C$. Thus the points of $\Sp\C$ are
equivalence classes of the form $[\chi]=\{\chi\comp\Si^n\mid n\in\bbZ\}$.

Take as an example the category of perfect complexes
$\D^b(\proj A)$ over a commutative ring $A$. A prime ideal
$\frp\in\Spec A$ with residue field $k(\frp)$ yields an irreducible
cohomological function
\[\chi_{k(\frp)}\colon\Ob\D^b(\proj A)\lto\bbN,\quad
X\mapsto\length_{k(\frp)}\Hom(X,k(\frp)),\]
where $k(\frp)$ is viewed as a complex concentrated in degree zero
\begin{thm}
  The map $\Spec A\to\Sp\D^b(\proj A)$ sending $\frp$ to
  $[\chi_{k(\frp)}]$ is injective and closed with respect to the Hochster
  dual of the Zariski topology on $\Spec A$.
\end{thm}

We prove this result in \S\ref{se:per} by analysing the Ziegler
spectrum \cite{Kr2000,Zi1984} of the category of perfect complexes. A
general method for computing the space $\Sp\C$ of cohomological
functions is to compute the Krull--Gabriel filtration
\cite{Ga1962,GJ1981} of the abelianisation $\Ab\C$ (see \cite{Fr,V} or
\S\ref{se:main_thm}).

A specific example is the algebra $k[\e]$ of dual numbers over a
field $k$. Each complex
\[X_n\colon\;\cdots\to 0\to k[\e]\xto{\e} k[\e]\xto{\e}\cdots\xto{\e}k[\e]\to
0\to \cdots\] of length $n$ corresponds to an isolated point in
$\Sp\D^b(\proj k[\e])$ and their closure yields exactly one extra point
corresponding to the residue field. Thus
\[\Sp\D^b(\proj k[\e])=\{[\chi_{X_n}]\mid n\in\bbN\}\cup\{[\chi_k]\}.\]
This example is of particular interest because the derived category
$\D^b(\proj k[\e])$ is discrete in the sense of Vossieck
\cite{Vo2001}, that is, there are no continuous families of
indecomposable objects. On the other hand, there are infinitely many
indecomposable objects, even up to shift.  The Krull--Gabriel
dimension explains this behaviour because it measures how far an
abelian category is away from being a length category. For instance, a
triangulated category $\C$ is locally finite (see \cite{Kr2012}
or \S\ref{se:per}) iff the Krull--Gabriel dimension of $\Ab\C$ equals
at most $0$.

\begin{prop}
  The abelianisation $\Ab\D^b(\proj k[\e])$ has Krull--Gabriel
  dimension equal to $1$.
\end{prop}

Note that the Krull--Gabriel dimension of the free abelian category
$\Ab A$ over an Artin algebra $A$ \cite{Gr1975} behaves differently;
it equals $0$ iff $A$ is of finite representation type by a result of
Auslander \cite{Au1974} and is greater than $1$ otherwise
\cite{He1997,Kr1998}.

As a final example, let us describe the cohomological functions for the
category $\coh\bbP^1_k$ of coherent sheaves on the projective line
over a field $k$.

\begin{prop}\label{pr:coh}
  The abelianisation $\Ab\D^b(\coh\bbP^1_k)$ has Krull--Gabriel
  dimension equal to $2$ and
\[\Sp\D^b(\coh\bbP^1_k)=\{[\chi_X]\mid X\in\coh\bbP^1_k\text{
  indecomposable}\}\cup\{[\chi_{k(t)}]\}.\]
\end{prop}

These examples illustrate in the triangulated context the following
representation theoretic paradigm:
\begin{center}
\begin{tabular}{lcr}
finite type&\;\;$\longleftrightarrow$\;\;&Krull--Gabriel dimension $= 0$\\
discrete type&$\longleftrightarrow$&Krull--Gabriel dimension $\le 1$\\
continuous families exist&$\longleftrightarrow$&Krull--Gabriel dimension $> 1$
\end{tabular}
\end{center}

\subsection*{Acknowledgements}
This work was inspired by lectures of David Eisenbud (Guanajuato, May
2012) and Peter Webb (Seattle, August 2012); see also \cite{We2013}. I
wish to thank Peter Webb for various helpful comments.

\section{Proof of the main theorem}\label{se:main_thm}

In this section Theorem~\ref{th:2} is proved. We deduce it from work
of Crawley-Boevey on endofinite objects in locally finitely presented
abelian categories \cite{CB1994a,CB1994b}. We begin with some preparations.

\subsection*{The abelianisation}
Following Freyd \cite[\S3]{Fr} and Verdier \cite[II.3]{V}, we consider
the \emph{abelianisation} $\Ab\C$ of $\C$ which is the abelian
category of additive functors $F\colon\C^\op\to\Ab$ into the category
$\Ab$ of abelian groups that admit a copresentation \[0\lto
F\lto\Hom(-,A)\lto\Hom(-,B).\] In addition, we work in the category
$\Mod \C$ of additive functors $\C^\op \to\Ab$.  This is a locally
finitely presented abelian category and the abelianisation $\Ab\C$
identifies with the full subcategory of finitely presented objects of
$\Mod\C$; see \cite{CB1994b} for details.

\subsection*{Additive functions}

Let us introduce the analogue of a cohomological function for the
abelianisation of $\C$.

\begin{defn}
  A function $\chi\colon\Ob\Ab\C\to\bbN$ is called
  \emph{additive}\footnote{The term \emph{additive} reflects condition
    (1), while (2) is added for technical reasons.}  provided that
\begin{enumerate}
\item  $\chi(F)=\chi(F')+\chi(F'')$ if $0\to F'\to F\to F''\to 0$ is an exact sequence, and
\item for each object $F$ there is some $n\in\bbZ$ such that $\chi(\Si^n F)=0$.
\end{enumerate}
\end{defn}

We show that additive and cohomological functions are closely related.

\begin{lem}\label{le:res}
  Restricting a function $\chi\colon\Ob\Ab\C\to\bbN$ to $\Ob\C$ by
  setting $\chi(C)=\chi(\Hom(-,C))$ gives a natural bijection between
\begin{enumerate}
\item[--]  the
  additive functions $\Ob\Ab\C\to\bbN$, and 
\item[--] the cohomological functions $\Ob\C\to\bbN$.
\end{enumerate}
\end{lem}
\begin{proof}
  Let $\chi\colon\Ob\Ab\C\to\bbN$ be an additive function.  An exact
  triangle $A\to B\to C\to $ in $\C$ yields in $\Ab\C$ an exact
  sequence \[\cdots\to \Hom(-,A)\to\Hom(-,B)\to\Hom(-,C)\to\Hom(-,\Si
  A)\to \cdots.\] Using this sequence it is easily checked that the
  restriction of $\chi$ to $\C$ is a cohomological function.

  Conversely, given a cohomological function $\chi\colon\Ob\C\to\bbN$,
  we extend it to a function $\Ob\Ab\C\to\bbN$ which again we denote
  by $\chi$. Fix $F$ in $\Ab\C$ with copresentation as above induced
  by an exact triangle $A\to B\to C\to $ in $\C$, and choose
  $n\in\bbZ$ such that $\chi(\Si^n(A\oplus B\oplus C))=0$.  Then
  define
\[\chi(F)=\begin{cases}
\sum_{i= 0}^n\big((-1)^i\chi(\Si^iA)-
(-1)^i\chi(\Si^iB)+(-1)^i\chi(\Si^iC)\big), &\text{if $n\ge 0$;}\\
\sum_{i= -1}^n\big((-1)^{i+1}\chi(\Si^iA)-
(-1)^{i+1}\chi(\Si^iB)+(-1)^{i+1}\chi(\Si^iC)\big), &\text{if $n< 0$.}
\end{cases}\] This gives a non-negative integer and does not depend on
$n$ since $\chi$ is a cohomological function. Also, $\chi(F)$ does not
depend on the choice of the exact triangle which presents $F$ by a
variant of Schanuel's lemma; see Lemma~\ref{le:shanuel}. Note that
$\chi(\Hom(-,C))=\chi(C)$ for each $C$ in $\C$. Standard arguments
involving resolutions show that $\chi$ is additive.

Clearly, restricting from $\Ab\C$ to $\C$ and extending from $\C$ to
$\Ab\C$ are mutually inverse operations.
\end{proof}

There is a parallel between functions and functors. The analogue of
Lemma~\ref{le:res} for functors is due to Freyd.
 
\begin{lem}[Freyd \cite{Fr}]\label{le:Freyd}
  Restricting a functor $F\colon(\Ab\C)^\op\to\Ab$ to $\C$ by
  setting $F(C)=F(\Hom(-,C))$ gives a natural bijection between
\begin{enumerate}
\item[--] the exact functors $(\Ab\C)^\op\to\Ab$, and
\item[--] the cohomological
  functors $\C^\op\to\Ab$.
\end{enumerate}
\end{lem}
\begin{proof}
  The inverse map sends a cohomological functor $H\colon\C^\op\to\Ab$
  to the exact functor $\Hom(-,H)\colon (\Ab\C)^\op\to\Ab$.
\end{proof}

An endofinite cohomological functor $H\colon\C^\op\to\Ab$ induces two
functions:
\begin{gather*}
\chi_H\colon\Ob\C\lto\bbN,\quad C\mapsto\length_{\End(H)}H(C),\\
\hat\chi_H\colon\Ob\Ab\C\lto\bbN,\quad F\mapsto\length_{\End(H)}\Hom(F,H).
\end{gather*}

\begin{lem}\label{le:fct}
  The function $\chi_H$ is cohomological and $\hat\chi_H$ is additive.
\end{lem}

\begin{proof}
  The functor $\Hom(-,H)\colon(\Ab\C)^\op\to\Ab$ is exact since $H$ is
  cohomological.  It follows that $\hat\chi_H$ is additive. The
  restriction of $\hat\chi_H$ to $\C$ equals $\chi_H$. Thus $\chi_H$
  is cohomological by Lemma~\ref{le:res}.
\end{proof}

\subsection*{Proof of the main theorem}
The proof of our main result is based on work of Crawley-Boevey, but
we provide some complementary material in an appendix.

\begin{proof}[Proof of Theorem~\ref{th:2}]
  An endofinite cohomological functor $H\colon\C^\op\to\Ab$ is a
  finite endolength injective object in $\Mod\C$, as defined in
  \cite{CB1994b}.  The term `finite endolength' refers to the fact
  that $\Hom(F,H)$ has finite length as $\End(H)$-module for each $F$
  in $\Ab\C$; see Lemma~\ref{le:fct}. The injectivity follows from
  the proof of \cite[Theorem~1.2]{Kr1999}, using that $\Ext^1(-,H)$
  vanishes on $\Ab\C$ for any cohomological functor $H$.

  In \cite[Theorem~3.5.2]{CB1994b} it is shown that each finite
  endolength object decomposes into a direct sum of indecomposable
  objects with local endomorphism rings; see
  \cite[Theorem~1.2]{Kr1999} for an alternative proof. This yields
  part (1).

  In \cite{CB1994a}, additive functions on locally finitely presented
  abelian categories are studied. In particular, there it is shown
  that every additive function on $\Ab\C$ can be written uniquely as
  a locally finite sum of irreducible additive functions. This proves
  part (2), in view of Lemma~\ref{le:res}. For an alternative proof,
  see Proposition~\ref{pr:add}.

  Finally, part (3) follows from the main theorem in \cite{CB1994a}
  which establishes the bijection between isomorphism classes of
  indecomposable finite endolength injective objects in $\Mod\C$ and
  irreducible additive functions $\Ob\Ab\C\to\bbN$.  For an
  alternative proof, see Proposition~\ref{pr:exact}, using the
  bijection between cohomological and exact functors from
  Lemma~\ref{le:Freyd}.
\end{proof}

\section{Properties of cohomological functions}

\subsection*{The correspondence between functors and functions}

The assignment $H\mapsto\chi_H$ between endofinite cohomological
functors and cohomological functions satisfies some weighted
additivity. For instance, $\chi_{H\oplus H'}=\chi_H+\chi_{H'}$
provided that $H$ and $H'$ have no common indecomposable summand, but
$\chi_{H\oplus H}=\chi_H$. We have the following concise formula.

\begin{prop}
  Let $H=\bigoplus_{i\in I} H_i$ be the decomposition of an endofinite
  cohomological functor $\C^\op\to\Ab$ into indecomposables. If
  $J\subseteq I$ is a subset such that $(H_i)_{i\in J}$ contains each
  isomorphism class from $(H_i)_{i\in I}$ exactly once then
  $\chi_H=\sum_{i\in J}\chi_{H_i}$.
\end{prop}
\begin{proof}
  Adapt the proofs of Propositions~4.5 and 4.6 in
  \cite{KR2001}. Alternatively, use Remark~\ref{re:sum}.
\end{proof}

\begin{rem}
  Let $\C$ be a $k$-linear category such that each morphism set in
  $\C$ has finite length as a $k$-module. Then we have two maps
  $\Ob\C\times\Ob\C\to\bbN$, taking $(X,Y)$ either to
  $\length_{k}\Hom(X,Y)$ or to $\length_{\End(Y)}\Hom(X,Y)$. While the
  first map preserves sums in both arguments, the second one does not
  in the second argument, but it satisfies the above `weighted
  additivity'.
\end{rem}

\subsection*{Duality}
The correspondence in Theorem~\ref{th:2} yields a remarkable duality
between cohomological functors $\C^\op\to\Ab$ and cohomological
functors $\C\to\Ab$. This follows from the fact that the definition of
a cohomological function $\Ob\C\to\bbN$ is self-dual; it is an
analogue of the \emph{elementary duality} between left and right
modules over a ring studied by Herzog \cite{He1993}.

  The duality links indecomposable endofinite cohomological functors
  $H\colon\C^\op\to\Ab$ and $H'\colon\C\to\Ab$ when
  $\chi_H=\chi_{H'}$. In that case 
  \[ (\End(H)/\rad\End(H))^\op\cong \End(H')/\rad\End(H')\] where
  $\rad A$ denotes the Jacobson radical of a ring $A$. This follows
  from Remark~\ref{re:end}.

The duality specialises to Serre duality when $\C$ is a Hom-finite
$k$-linear category with $k$ a field. More precisely, if
$F\colon\C\to\C$ is a Serre functor \cite{RV2002} and $D=\Hom(-,k)$, then
\[\Hom(-,FX)\cong D\Hom(X,-)\]
for each object $X$, and therefore $\chi_{\Hom(-,FX)}=\chi_{\Hom(X,-)}$.

\subsection*{The space of cohomological functions}

Consider the set of irreducible cohomological functions $\Ob\C\to\bbN$
and identify this via Lemma~\ref{le:res} with a subspace of
$\Sp\Ab\C$, endowed with the Ziegler topology, as explained in
Proposition~\ref{pr:top}.  The quotient
\[\Sp\C=\{\chi\colon \Ob\C\to\bbN\mid \text{$\chi$
  irreducible and cohomological}\}/\Si\] with respect to the action of
the suspension is by definition the \emph{space of cohomological
  functions} on $\C$. Thus the points of $\Sp\C$ are
equivalence classes of the form $[\chi]=\{\chi\comp\Si^n\mid n\in\bbZ\}$.

The construction of this space is functorial with respect to certain
functors.  Let $f\colon\C\to\D$ be an exact functor between
triangulated categories. Given $[\chi]$ in $\Sp\D$ the composite
$\chi\comp f$ is cohomological but need not be irreducible. Thus
$f$ induces a continuous map $\Sp\D\to\Sp\C$ provided that
irreducibility is preserved. For instance a quotient functor
$\C\to\C/\B$ with respect to a triangulated subcategory
$\B\subseteq\C$ has this property; it induces a
homeomorphism \[\Sp\C/\B
\stackrel{\sim}\longrightarrow\{[\chi]\in\Sp\C\mid\chi(\B)=0\}.\]

Before we discuss specific examples, let us give one general result.
Let $k$ be a field and $\C$ be a $k$-linear triangulated category such
that for each pair of objects $X,Y$ we have $\Hom(X,\Si^nY)=0$ for
some $n\in\bbZ$.  Suppose that all morphism spaces are finite
dimensional and that $\C$ is idempotent complete. Suppose also that
$\C$ admits a Serre functor \cite{RV2002}. Denote for each object $X$
by $\chi_X$ the cohomological function corresponding to $\Hom(-,X)$.

\begin{prop}\label{pr:isolated}
  A point in $\Sp\C$ is isolated if and only if it equals $[\chi_X]$
  for some indecomposable object $X$. Moreover, the isolated points
  form a dense subset of $\Sp\C$.
\end{prop}
\begin{proof}
  Each indecomposable object $X$ fits into an Auslander--Reiten
  triangle $X\to Y\to Z\to $ in $\C$, by \cite[Theorem~I.2.4]{RV2002}.
  Such a triangle provides in $\Ab\C$ the following copresentation of
  a simple object $S_X$
\[0\lto S_X\lto\Hom(-,X)\lto\Hom(-,Y).\] Thus
$(S_X)=\{\chi_X\}$ is a basic open set for each indecomposable object
$X$.

Now let $(F)$ be a non-empty basic open set with $F\in\Ab\C$. The
functor $F$ admits a copresentation
\[0\lto F\lto\Hom(-,X)\lto\Hom(-,Y)\]  with  an
indecomposable direct summand $X'\subseteq X$ such that
$\Hom(F,\Hom(-,X'))\neq 0$. Thus $\chi_{X'}$ belongs to $(F)$. It
follows that each non-empty open subset of $\Sp\C$ contains a point of
the form $[\chi_X]$.
\end{proof}

The space of cohomological functions may be empty as the following
example shows.

\begin{exm}
  Let $A$ be a ring without invariant basis number, for example the
  endomorphism ring of an infinite dimensional vector space. Then
  there is no non-zero endofinite cohomological functor
  $H\colon\D^b(\proj A)^\op\to\Ab$. To see this, observe that $H(\Si^n
  A)$ is a finite endolength $A$-module for all $n\in\bbZ$, and
  therefore the zero module \cite[\S4.7]{CB1992}.
\end{exm}

\section{Examples: perfect complexes}\label{se:per}

We compute the space $\Sp\C$ of cohomological functions in some
examples, for instance when $\C$ is the triangulated category of
perfect complexes over some ring. It is convenient to view $\Sp\C$ as
a subspace of the spectrum $\Zsp\C$, as defined in
Appendix~\ref{se:Zsp}.

The problem of computing the space of cohomological functions is
reduced to the study of the Krull--Gabriel filtration of the
abelianisation $\Ab\C$. This filtrations yields a dimension. For an
abelian category $\A$, the Krull--Gabriel dimension $\KGdim\A$ is an
invariant which measures how far $\A$ is away from being a length
category; see Appendix~\ref{se:Zsp}.

\subsection*{Modules}
Let $A$ be a ring. We write $\Mod A$ for the category of $A$-modules,
$\mod A$ for the full subcategory of finitely presented ones, and
$\proj A$ for the full subcategory of finitely generated
projectives. 

Following \cite{Gr1975}, we consider the \emph{free abelian
  category}\footnote{The category $\Ab A$ is the opposite of the
  category of functors $F\colon\mod A\to\Ab$ that admit a presentation
  $\Hom(Y,-)\to\Hom(X,-)\to F\to 0$, and $A$ (viewed as category with
  a single object) embeds via $A\mapsto\Hom(A,-)$. Any additive
  functor $A\to\A$ to an abelian category $\A$ extends uniquely to an
  exact functor $\Ab A\to\A$.}  over $A$ and denote it by $\Ab
A$. Thus the category of $A$-modules identifies with the category of
exact functors $(\Ab A)^\op\to\Ab$.

We consider the derived category $\D(\Mod A)$ and write $\per A$ for
the full subcategory $\D^b(\proj A)$ of \emph{perfect complexes}.

The \emph{Ziegler spectrum} $\Zsp A$ of $A$ is by definition the set
of isomorphism classes of indecomposable pure-injective $A$-modules
with the topology introduced by Ziegler \cite{Pr2009,Zi1984}. We
identify $\Zsp A$ with the spectrum $\Zsp\Ab A$ of the abelian
category $\Ab A$; see Appendix~\ref{se:Zsp}.

Note that the spectrum $\Zsp\per A$ identifies with a quotient  of the Ziegler
spectrum of the compactly generated triangulated category $\D(\Mod A)$
introduced in \cite{Kr2000} and further investigated in \cite{GP2005}.

Let us compare $\Zsp A$ and $\Zsp\per A$. We consider the
natural inclusion $i\colon A\to \per A$ which extends to an exact
functor $i^*\colon\Ab A\to\Ab\per A$ by the universal property of $\Ab A$.

\begin{lem}\label{le:ZgA}
  Let $\Sc_0\subseteq\Ab\per A$ be the Serre subcategory
  generated by the representable functors $\Hom(-,\Si^n A)$ with
  $n\neq 0$. Then the composite 
  \[\Ab A\xto{i^*}\Ab\per A\twoheadrightarrow (\Ab\per A)/\Sc_0\]
is an equivalence.
\end{lem}
\begin{proof}
  The cohomological functors $H\colon(\per A)^\op\to\Ab$ annihilating
  $\Si^n A$ for all $n\neq 0$ identify with $\Mod A$, by taking $H$
  to $H(A)$. Thus the exact functors $(\Ab\per A)^\op\to\Ab$
  annihilating $\Sc_0$ identify with $\Mod A$, by
  Lemma~\ref{le:Freyd}. From this the assertion follows.
\end{proof}

We view an $A$-module $X$ as a complex concentrated in degree zero and
denote by $H_X$ the corresponding cohomological functor
$\Hom(-,X)\colon (\per A)^\op\to\Ab$. It is convenient to identify
$H_X$ with the exact functor \[\Hom(-,H_X)\colon (\Ab\per
A)^\op\lto\Ab.\]

\begin{prop}\label{pr:ZspA}
The assignment $X\mapsto [H_X]$ induces an injective and continuous map
$\p\colon \Zsp A\to\Zsp\per A$; its image is a closed subset.
\end{prop}
\begin{proof}
  We apply Lemma~\ref{le:ZgA}. The composite \[f\colon\Ab\per
  A\twoheadrightarrow(\Ab\per A)/\Sc_0\xto{\sim}\Ab A\] identifies
  $\Zsp A$ with a closed subset of $\Zsp\Ab\per A$ which we denote by
  $\U$. Viewing an $A$-module $X$ as an exact functor
  $(\Ab A)^\op\to\Ab$, we have $X\comp f=H_X$.  Note that the subsets
  $\Si^n \U$, $n\in\bbZ$, are pairwise disjoint. It follows that $\p$
  is injective.

  Next observe that
  \[\bigcup_{n\in\bbZ}\Si^n\U=(\Zsp\Ab\per A)\setminus\bigcup_{(r,s)}\U_{r,s}\]
where $(r,s)$ runs through all pairs of integers $r\neq s$ and 
\[\U_{r,s}=\{F\in\Zsp\Ab\per A\mid F(\Si^rA)\neq 0\neq F(\Si^s A)\}.\]
Thus the image of $\p$ is closed.

If a closed subset $\V\subseteq \Zsp\per A$ consists of all exact
functors vanishing on the Serre subcategory $\Sc\subseteq\Ab\per A$,
then $\p^{-1}(\V)$ consists of all exact functors vanishing on
$f(\Sc)$. Thus $\p$ is continuous.
\end{proof}

\subsection*{Hereditary rings}
A complex of $A$-modules can be written as a direct sum of stalk
complexes when $A$ is hereditary. This has some useful consequences
which are collected in the following proposition.  

\begin{prop}\label{pr:hered}
For a hereditary ring $A$  the following holds.
\begin{enumerate}
\item The map $\phi\colon\Zsp A\to\Zsp\per A$ taking $X$ to $[H_X]$ is a
homeomorphism.
\item $\KGdim\Ab\per A=\KGdim\Ab A$.
\end{enumerate}
\end{prop}
\begin{proof}
(1)  We apply Proposition~\ref{pr:ZspA}.  Each indecomposable complex of
  $A$-modules is concentrated in a single degree since $A$ is
  hereditary. This observation yields a disjoint union
\[\Zsp\Ab\per A=\bigcup_{n\in\bbZ}\U_n\qquad\text{with}
\qquad\U_n=\{F\in \Zsp\Ab\per A\mid F(\Si^n A)\neq 0\},\] and each
$\U_n$ is homeomorphic to $\Zsp A$. An open subset $\V\subseteq \Zsp A$
identifies with an open subset $\V_n\subseteq \U_n$ and therefore with
an open subset of $\Zsp\per A$ via $\p$. It follows that $\p$ is open
and therefore a homeomorphism by Proposition~\ref{pr:ZspA}.

(2) We apply Lemma~\ref{le:KGdim} and use the family of quotient
functors \[(\Ab\per A\twoheadrightarrow (\Ab\per A)/\Sc_n\xto{\sim}\Ab
A)_{n\in\bbZ}\] from Lemma~\ref{le:ZgA}.
\end{proof}

\subsection*{Commutative rings}

Let $A$ be a commutative ring and $\Spec A$ the set of prime ideals.
We endow $\Spec A$ with the dual of the Zariski topology in the sense
of Hochster \cite{Ho1969}.  A prime ideal $\frp$ with residue field
$k(\frp)$ yields an irreducible cohomological function
\[\chi_{k(\frp)}\colon\Ob\per A\lto\bbN,\quad
X\mapsto\length_{k(\frp)}\Hom(X,k(\frp)).\]
As before, we identify $\chi_{k(\frp)}$ with $k(\frp)$. In particular,
$\chi_{k(\frp)}$ is irreducible since  $k(\frp)$ is indecomposable.

\begin{thm}
  The map $\rho\colon\Spec A\to\Sp\per A$ sending $\frp$
  to $[\chi_{k(\frp)}]$ is injective and closed with respect to the
  Hochster dual of the Zariski topology on $\Spec A$.
\end{thm}

\begin{proof}
  The injectivity is clear since different primes $\frp,\frq$ yield
  non-isomorphic functors $\Hom(-,k(\frp))$ and $\Hom(-,k(\frq))$.

To show that the image  $\Im\rho$ is closed, observe first that 
\[\U=\{[\chi]\in\Sp\per A\mid \chi(\Si^n A)\neq 0\text{ for
  at most one }n\in\bbZ\}\] is closed by Proposition~\ref{pr:ZspA}.
Also, \[\U_1=\{[\chi]\in\Sp\per A\mid \chi(\Si^n A)\le 1\text{ for all
}n\in\bbZ\}\] is closed by Lemma~\ref{le:finite}.  The indecomposable
$A$-modules $X$ with $\length_{\End(X)}X\le 1$ are precisely the
residue fields $k(\frp)$; see \cite[\S4.7]{CB1992}.  Thus
$\Im\rho=\U\cap \U_1$ is closed.

Given a closed subset $\V\subseteq\Spec A$, we need to show that
$\rho(\V)$ is closed.  It follows from Thomason's classification of
thick subcategories \cite[Theorem~3.15]{Th1997} that there is a thick
subcategory $\C$ of $\per A$ with $\frp\in\V$ iff $\Hom(X,k(\frp))=0$
for all $X\in\C$. The latter condition means $\chi_{k(\frp)}(X)=0$ for
all $X\in\C$.  Thus $\{[\chi_{k(\frp)}]\mid \frp\in\V\}$ is Ziegler
closed.
\end{proof}

\begin{rem}
This result generalises to schemes that are quasi-compact and quasi-separated.
\end{rem}

\subsection*{Krull--Gabriel dimension zero}

Following \cite{Kr2012}, a triangulated category $\C$ is \emph{locally
  finite} if its abelianisation $\Ab\C$ is a length category, which
means that $\KGdim\Ab\C\le 0$. Set $\chi_X=\chi_{\Hom(-,X)}$ for
$X\in\C$ when $\Hom(-,X)$ is endofinite.

\begin{prop}\label{pr:KGdim=0}
  Let $\C$ be a locally finite and idempotent complete triangulated
  category. Suppose  for each pair of objects $X,Y$ that
  $\Hom(X,\Si^nY)=0$ for some $n\in\bbZ$. Then
\[\Sp\C=\{[\chi_X]\mid X\in\C\text{ indecomposable}\}.\]
\end{prop}
\begin{proof}
Let $X$ be an  object in $\C$. Then we have for each object $C$ 
\[\length_{\End(X)}\Hom(C,X)\le \length_{\Ab\C}\Hom(-,C)<\infty\]
by \cite[Theorem~2.12]{Au1974}; see also
Remark~\ref{re:sum}. Thus $\Hom(-,X)$ is endofinite.

Let $\chi\colon\Ob\C\to\bbN$ be an irreducible cohomological
function and $\hat\chi\colon\Ob\Ab\C\to\bbN$ its extension to
$\Ab\C$. Then $\hat\chi(S)\neq 0$ for some simple object $S$. There is
an indecomposable object $X$ in $\C$ with $\Hom(S,\Hom(-,X))\neq 0$,
and it follows that $\Hom(-,X)$ is an injective envelope. Thus $\chi=\chi_X$.
\end{proof}

\begin{exm}
  Let $k$ be a field and $\Ga$ a quiver with underlying diagram of
  Dynkin type.  Denote by $\rep (\Ga,k)$ the category of finite
  dimensional $k$-linear representations of $\Ga$. The indecomposable
  endofinite cohomological functors $\D^b(\rep (\Ga,k))^\op\to\Ab$ are
  precisely (up to isomorphism) the representable functors
  $\Hom(-,\Si^nX)$ with $X$ an indecomposable object in $\rep (\Ga,k)$
  (viewed as complex concentrated in degree zero) and $n\in\bbZ$.
\end{exm}

Examples of triangulated categories $\C$ with $\KGdim\Ab\C=1$ show
that the description of $\Sp\C$ in Proposition~\ref{pr:KGdim=0} does
not hold more generally.

\subsection*{Krull--Gabriel dimension one}

Let $k$ be a field. We give an example of a finite dimensional
$k$-algebra $A$ such that the Krull--Gabriel dimension of $\Ab\per A$
equals $1$. This gives some evidence for the following conjecture.

\begin{conj}
  Let $A$ be a finite dimensional $k$-algebra. Then the
  Krull--Gabriel dimension of $\Ab\per A$ equals $0$ or $1$ if and
  only if $\per A$ is a discrete derived category in the sense of
  \cite{Vo2001}.
\end{conj}
 
\begin{prop}\label{pr:dim1}
Let $k[\e]$ be the algebra of dual numbers. Then
  \[\KGdim\Ab\per k[\e]=1.\]
\end{prop}
\begin{proof}
  Set $\C=\per k[\e]$ and write $H_X=\Hom(-,X)$ for each $X$ in
  $\C$. The indecomposable objects are the complexes
\[X_{n,r}\colon\;\cdots\to 0\to k[\e]\xto{\e}k[\e]\xto{\e}\cdots\xto{\e}k[\e]\to
0\to \cdots\]
concentrated in degrees $n, n+1,\cdots, n+r$ and parametrised by pairs
$(n,r)$ in
$\bbZ\times\bbN$. The Auslander--Reiten triangles are of the form
\[X_{n+1,r}\lto X_{n+1,r-1}\oplus X_{n,r+1}\lto X_{n,r}\stackrel{\e}\lto  X_{n,r}\] with
$X_{n,-1}=0$. Note that the last morphism is induced by multiplication
with $\e$. Such a triangle induces in $\Ab\C$ an
exact sequence
\begin{equation}\label{eq:AR}
0\to S_{n+1,r}\to H_{X_{n+1,r}}\to H_{X_{n+1,r-1}}\oplus
H_{X_{n,r+1}}\to H_{X_{n,r}}\to S_{n,r}\to 0
\end{equation}
with simple end terms. 

Fix $n\in\bbZ$.  We claim that the Hasse diagram of the lattice of
subobjects of $H_{X_{n,0}}$ has the following form.
\[\xymatrix@R=.45pc{
  \scriptscriptstyle{\bullet} \ar@{-}[d]\\
  \scriptscriptstyle{\bullet}\ar@{-}[d]\\
  \scriptscriptstyle{\bullet}  \ar@{-}[d]\\  \ar@{-}[d]\\
  \vdots \\ \ar@{-}[d]\\ \scriptscriptstyle{\bullet}\ar@{-}[d]\\
  \scriptscriptstyle{\bullet}\ar@{-}[d]\\ \scriptscriptstyle{\bullet}
}\] To prove this, consider the sequence of morphisms
\[\cdots\lto X_{n,2}\lto X_{n,1}\lto X_{n,0}\]
given by the Auslander--Reiten triangles.  For each $t\ge 0$, the
composite $\p_{n,t}\colon X_{n,t}\to X_{n,0}$ induces a morphism
$H_{\p_{n,t}}$ in $\Ab\C$ and its image is the unique subobject
$U\subseteq H_{X_{n,0}}$ such that $H_{X_{n,0}}/U$ has length
$t$. This explains the upper half of the Hasse diagram; the form of
the lower half then follows by Serre duality. More precisely, Serre
duality yields an equivalence $\C^\op\xto{\sim}\C$ which is the
identity on objects. It extends to an equivalence $(\Ab
\C)^\op\xto{\sim}\Ab (\C^\op)\xto{\sim}\Ab\C$ which induces a
bijection between subobjects and quotient objects of $H_{X_{n,0}}$.
It remains to show that $H_{X_{n,0}}$ has no further subobjects. To
see this, let $V\subseteq H_{X_{n,0}}$ be a subobject; it is the image
of some morphism $H_\p\colon H_X\to H_{X_{n,0}}$. We may assume
$\p\neq 0$ and that $X$ is indecomposable.  The property of the
Auslander--Reiten triangle for $X_{n,0}$ implies that the endomorphism
$\e\colon X_{n,0}\to X_{n,0}$ factors through $\p$ via a morphism
$\p'\colon X_{n,0}\to X$. Thus $\p$ and $\p'$ yield in degree zero
endomorphisms of $k[\e]$, and exactly one of them is an
isomorphism. If $\p^0$ is an isomorphism, then $H_{X_{n,0}}/V$ has
finite length; otherwise $V$ is of finite length.

The form of the lattice of subobjects implies that $H_{X_{n,0}}$ is a
simple object in $(\Ab\C)/(\Ab\C)_0$.  Using induction on $r$, the
sequence \eqref{eq:AR} shows that $H_{X_{n,r}}$ has length $r+1$. Thus
$(\Ab\C)_1=\Ab\C$.
\end{proof}

\begin{cor}\label{co:dim1}
We have $\Sp\per k[\e]=\{[\chi_{X_{0,r}}]\mid r\in\bbN\}\cup\{[\chi_k]\}$.
\end{cor}
\begin{proof}
Set $\C=\per k[\e]$.  The Krull--Gabriel filtration of $\Ab\C$ yields a filtration
  of $\Zsp\C$ by Proposition~\ref{pr:ZspKGdim}. Thus the
  points of $\Zsp \C$ correspond to the simple objects in
  $\Ab\C$ and $(\Ab\C)/(\Ab\C)_0$.  These
  simple objects are described in the proof of
  Proposition~\ref{pr:dim1}.  The simples in $\Ab\C$
  correspond to the indecomposable objects in $\C$ and yield
  isolated points; see also Proposition~\ref{pr:isolated} and its
  proof. The simples in $(\Ab\C)/(\Ab\C)_0$ correspond
  to the complexes with $k$ concentrated in a single degree. Thus all
  points in $\Zsp\C$ are endofinite, and this yields the
  description of $\Sp\C$.
\end{proof}

\begin{rem}
Note that $\KGdim\Ab A\neq 1$ for any Artin algebra $A$ \cite{He1997,Kr1998}.
\end{rem}

\subsection*{Krull--Gabriel dimension two}

Let $k$ be a field and consider the Kronecker algebra
$A=\smatrix{k&k^2\\0&k}$. Work of Geigle \cite{Ge1985} shows that the
Krull--Gabriel dimension of $\Ab A$ equals $2$.  Thus $\KGdim\Ab\per
A=2$ by Proposition~\ref{pr:hered}, since $A$ is hereditary.

Now let $\coh\bbP^1_k$ be the category of coherent sheaves on the
projective line over $k$.  There is a well-known derived equivalence
\[\RHom(T,-)\colon\D^b(\coh \bbP^1_k)\stackrel{\sim}\longrightarrow\D^b(\mod A)\]
given by $T=\mathcal O(0)\oplus \mathcal O(1)$, and we use this to
establish the description of the cohomological functions on
$\coh\bbP^1_k$ stated in the introduction.

\begin{proof}[Proof of Proposition~\ref{pr:coh}]
  We have 
\[\KGdim\Ab\D^b(\coh\bbP^1_k)=\KGdim\Ab\per A =\KGdim\Ab A=2\] by
Proposition~\ref{pr:hered} and \cite[Theorem~4.3]{Ge1985}. This yields
an explicit description of the points in $\Zsp\D^b(\coh\bbP^1_k)$
which is parallel to that given in Corollary~\ref{co:dim1}. More
explicitly, the indecomposable endofinite cohomological functors
$\D^b(\coh\bbP^1_k)^\op\to\Ab$ are precisely the representable
functors $\Hom(-,\Si^nX)$ with $X$ an indecomposable object in
$\coh\bbP^1_k$ or $X=k(t)$ the function field and $n\in\bbZ$.
\end{proof}
 
\begin{appendix}

\section{Schanuel's lemma for triangulated categories}

Let $\C$ be a triangulated category. An exact triangle
$A\to B\to C\to $ induces a \emph{presentation} of a functor
$F\colon\C^\op\to\Ab$ provided there exists an exact sequence
\[\Hom(-,B)\lto\Hom(-,C)\lto F\lto 0.\]
Two exact triangles are called \emph{homotopy
  equivalent}\footnote{This notion is consistent with the homotopy
  relation introduced in \cite[Section~1.3]{Ne2001}.} if they induce
presentations of the same functor.

\begin{lem}\label{le:shanuel}
  Let $A\to B\to C\to $ and $A'\to B'\to C'\to $ be two homotopy
  equivalent exact triangles. Then $A\oplus B'\oplus C\cong A'\oplus
  B\oplus C'$.
\end{lem}
\begin{proof}
The triangles induce exact sequences
\[0\to\Si^{-1}F\to\Hom(-,A)\to\Hom(-,B)\to\Hom(-,C)\to F\to 0\]
and
\[0\to\Si^{-1}F\to\Hom(-,A')\to\Hom(-,B')\to\Hom(-,C')\to F\to 0\]
which represent the same class in $\Ext^3(F,\Si^{-1}F)$, since the
presentations induce a morphism between both triangles.  Now apply the
variant of Schanuel's lemma which is given below.
\end{proof}

\begin{lem}
  Let \[0\to M\to P_r\to\cdots \to P_1\to P_0\to N\to 0\] and \[0\to
  M\to Q_r\to\cdots \to Q_1\to Q_0\to N\to 0\] be exact sequences in
  some abelian category which represent the same class in
  $\Ext^{r+1}(N,M)$. If all $P_i$ and $Q_i$ are projective, then
\[\bigoplus_{i\ge 0}(P_{2i}\oplus Q_{2i+1})\cong \bigoplus_{i\ge 0}(P_{2i+1}\oplus Q_{2i}).\]
\end{lem}
\begin{proof}
  We use induction on $r$. The case $r=0$ is clear and we suppose that
  $r>0$. The pullback of $\eta\colon P_0\to N$ and $\theta\colon
  Q_0\to N$ induces an exact sequence
\[0\lto K\lto P_0\oplus Q_0\lto N\lto 0\] with $Q_0\oplus \Ker \eta\cong K\cong
P_0\oplus \Ker\theta$, by Schanuel's lemma. Adding complexes of the form
$Q_0\xto{\id}Q_0$ and $P_0\xto{\id}P_0$ yields two exact sequences
\[0\to M\to P_r\to\cdots \to P_2\to P_1\oplus Q_0\to K\to 0\]
and \[0\to M\to Q_r\to\cdots \to Q_2\to Q_1\oplus P_0\to K\to 0\]
which represent the same class in $\Ext^{r}(K,M)$. Now the assertion
follows from the induction hypothesis.
\end{proof}

\section{Additive functions}

Let $\A$ be an abelian category. A function $\chi\colon\Ob\A\to\bbN$
is called \emph{additive} if $\chi(X)=\chi(X')+\chi(X'')$ for each
exact sequence $0\to X'\to X\to X''\to 0$.

We give a quick proof of the following result using the localisation
theory for abelian categories \cite{Ga1962}.

\begin{prop}[Crawley-Boevey \cite{CB1994a}]\label{pr:add}
  Every additive function $\Ob\A\to\bbN$ can be written uniquely as a
  locally finite sum of irreducible additive functions.
\end{prop}
\begin{proof}
  Fix an additive function $\chi\colon\Ob\A\to\bbN$. The objects $X$
  satisfying $\chi(X)=0$ form a Serre subcategory of $\A$ which we
  denote by $\Sc_\chi$. The quotient category $\A/\Sc_\chi$ is an
  abelian length category since the length of each object $X$ is
  bounded by $\chi(X)$. Let $\Sp\chi$ (the \emph{spectrum} of $\chi$)
  denote a representative set of simple objects in $\A/\Sc_\chi$. For
  each $S$ in $\Sp\chi$ let $\Sc_S$ denote the Serre subcategory of
  $\A$ formed by all objects $X$ such that a composition series of $X$
  in $\A/\Sc_\chi$ has no factor isomorphic to $S$. Define
  $\chi_S\colon\Ob\A\to\bbN$ by sending $X$ to the length of $X$ in
  $\A/\Sc_S$. From the construction it follows
  that 
\begin{equation}\label{eq:chi}
    \chi=\sum_{S\in\Sp\chi}\chi(S)\chi_S.
  \end{equation}
  We claim that each $\chi_S$ is irreducible and that the above
  expression is unique.  To see this, write $\chi=\chi'+\chi''$ as a
  sum of two additive functions. This implies
  $\Sc_\chi\subseteq\Sc_{\chi'}$, and if $\chi'\neq 0$, then for some
  $S\in\Sp\chi$ the object $S$ is non-zero in $\A/\Sc_{\chi'}$. In
  that case $\chi_S$ arises as a summand of $\chi'$ with multiplicity
  $\chi'(S)$.
\end{proof}

Now suppose that $\A$ is \emph{essentially small}.  Thus the isomorphism
classes of objects in $\A$ form a set. An exact functor
$F\colon\A^\op\to\Ab$ is called \emph{endofinite} if $F(X)$ has finite
length as $\End(F)$-module for each object $X$. An endofinite exact
functor $F$ induces an additive function
\[\chi_F\colon\Ob\A\lto\bbN,\quad X\mapsto\length_{\End(F)}F(X).\]

\begin{prop}\label{pr:exact}
The assignment $F\mapsto \chi_F$ induces a bijection between the
  isomorphism classes of indecomposable endofinite exact
  functors $\A^\op\to\Ab$ and the irreducible additive functions
  $\Ob\A\to\bbN$.
\end{prop}
\begin{proof}
  We construct the inverse map. Let $\chi\colon\Ob\A\to\bbN$ be an
  irreducible additive function. Following the proof of
  Proposition~\ref{pr:add}, we consider the Serre subcategory
  $\Sc_\chi$ of $\A$ consisting of the objects $X$ satisfying
  $\chi(X)=0$. The quotient category $\B=\A/\Sc_\chi$ is an abelian
  length category, and $\chi(X)$ equals the length of $X$ in $\B$ for
  each object $X$, since $\chi$ is irreducible.  Now consider the
  abelian category $\Lex(\B^\op,\Ab)$ of left exact functors
  $\B^\op\to\Ab$; see \cite{Ga1962} for details. The Yoneda functor
  \[\B\lto \Lex(\B^\op,\Ab),\quad X\mapsto H_X=\Hom(-,X)\]
  identifies $\B$ with the full subcategory of finite length
  objects. There is a unique simple object in $\Lex(\B^\op,\Ab)$
  since $\chi$ is irreducible, and we denote by $F$ its injective
  envelope. It follows that $F$ is indecomposable, and the injectivity
  implies that $F$ is exact.  For each $X$ in $\B$ we have
\[\length_{\End(F)}F(X)=\length_{\End(F)}\Hom(H_X,F)=\length_\B X=\chi(X)\]
since each finitely generated $\End(F)$-submodule of $\Hom(H_X,F)$ is
of the form $\Hom(H_X/H_{X'},F)$ for some subobject $X'\subseteq X$.
Let $F'\colon\A^\op\to\Ab$ be the composite of $F$ with the quotient
functor $\A\to\B$ and observe that $\End(F')\cong\End(F)$.  Then $F'$
has the desired properties: it is indecomposable endofinite exact and
$\chi_{F'}=\chi$.

It remains to show for an indecomposable endofinite exact functor
$F\colon\A^\op\to\Ab$ that the function $\chi_F$ is irreducible.  Set
$\B=\A/\Sc_{\chi_F}$ and view $F$ as an exact functor $\B^\op\to\Ab$.
Note that $\Hom(H_S,F)=F(S)\neq 0$ for each simple object $S$ in
$\B$. The indecomposability of $F$ implies that all simple objects in
$\B$ are isomorphic, and the equation \eqref{eq:chi} then implies that
$\chi$ is irreducible since for each simple object $S$
\[\chi_F(S)=\length_{\End(F)}F(S)=\length_{\End(F)}\Hom(H_S,F)=\length_\B
S=1.\qedhere\]
\end{proof}

\begin{rem}\label{re:end}
Let $F\colon\A^\op\to\Ab$ be an indecomposable endofinite exact
functor and $S$ the corresponding simple object in $\A/\Sc_{\chi_F}$.
Then the endomorphism ring $\End(F)$ is local and
\[\End(F)/\rad\End(F)\cong\End(S)\]
since $F$ identifies with an injective envelope of $S$. Here, $\rad
A$ denotes the Jacobson radical of a ring $A$.
\end{rem}

\begin{rem}\label{re:sum}
Let $F\colon\A^\op\to\Ab$ be an exact
functor and $\B=\A/\Sc_{F}$, where $\Sc_F$ denotes the Serre subcategory of
objects $X$ satisfying $F(X)=0$. For each object $X$ in $\A$ we have
\[\chi_F(X)=\length_{\End(F)}F(X)=\length_{\End(F)}\Hom(H_X,F)=\length_\B X,\]
and this can be used to compute $\sum_i\chi_{F_i}$ for any
decomposition $F=\bigoplus_i F_i$ into exact functors.
\end{rem}

Let $\Sp\A$ denote the set of irreducible additive functions
$\Ob\A\to\bbN$. Following \cite[\S4]{Kr1997}, we define on $\Sp\A$ the
\emph{Ziegler topology}; the basic open sets are of the form
\[(X)=\{\chi\in\Sp\A\mid\chi(X)\neq 0\},\qquad X\in\Ob\A.\] 

\begin{prop}\label{pr:top}
  The set $\Sp\A$ of irreducible additive functions $\Ob\A\to\bbN$
  forms a topological space which satisfies the $T_1$-axiom, that is,
  $\{\chi\}$ is closed for each $\chi\in\Sp\A$.
\end{prop}
\begin{proof}
  We identify each irreducible additive function $\Ob\A\to\bbN$ with
  an indecomposable injective object in $\Lex(\A^\op,\Ab)$, as in
  Proposition~\ref{pr:exact} and its proof. Thus
  \cite[Lemma~4.1]{Kr1997} applies, and the argument given there shows
  that for two objects $X_1,X_2$ in $\A$, the set $(X_1)\cap (X_2)$
  can be written as union of basic open sets.

  A singleton $\{\chi\}$ is closed since $\chi$ is the only irreducible
  function satisfying $\chi(X)=0$ for all $X\in\Sc_\chi$.
\end{proof}

The space $\Sp\A$ of additive functions identifies via
Proposition~\ref{pr:exact} with a subspace of $\Zsp\A$ which is
discussed in the subsequent Appendix~\ref{se:Zsp}.

\begin{lem}\label{le:finite}
  Let $X$ be an object in $\A$ and $n\ge 0$. Then
\[\U_{X,n}=\{\chi\in\Sp\A\mid\chi(X)\le n\}\]
is a closed subset of $\Sp\A$.
\end{lem}
\begin{proof}
For a chain of subobjects
\[\p\colon \;0=X_0\subseteq X_1\subseteq\ldots\subseteq X_{n+1}=X\] set
\[\U_\p=\bigcup_{i=0}^n\{\chi\in\Sp\A\mid \chi(X_{i+1}/X_i)=0\},\]
and let $\U=\bigcap_\p\U_\p$ where $\p=(X_i)_{0\le i\le n+1}$
runs through all such chains. This set is closed by construction, and
it follows from Remark~\ref{re:sum} that $\U=\U_{X,n}$.
\end{proof}

\section{The spectrum of an abelian category}\label{se:Zsp}

Let $\A$ be an essentially small abelian category. We consider the
category of exact functors $\A^\op\to\Ab$. This category inherits an
exact structure from $\Ab$ and we denote by $\Zsp\A$ the set of
isomorphism class of indecomposable injective objects.  Note that
$\Zsp\A$ equals the spectrum of the Grothendieck abelian category
$\Lex(\A^\op,\Ab)$ of left exact functors $\A^\op\to\Ab$ in the sense
of \cite[Chap.~IV]{Ga1962}.  Following \cite[\S4]{Kr1997}, we define
on $\Zsp\A$ the \emph{Ziegler topology}; the basic open sets are of
the form
\[(X)=\{F\in\Zsp\A\mid F(X)\neq 0\},\qquad X\in\Ob\A.\] 

\begin{lem}\label{le:Zsp}
The assignment
\[\Zsp\A\supseteq\U\longmapsto\{X\in\A\mid F(X)=0\text{ for
  all }F\in\U\}\] induces an inclusion reversing bijection between the
closed subsets of $\Zsp\A$ and the Serre subcategories of $\A$. In
particular, $\Zsp\A$ is quasi-compact iff $\A$ admits a generator,
that is, an object not contained in any proper Serre subcategory of $\A$.
\end{lem}
\begin{proof}
See Theorem~4.2 and Corollary~4.5 in \cite{Kr1997}. 
\end{proof}

The construction of this space is functorial with respect to certain
functors.  Let $f\colon\A\to\B$ be an exact functor between abelian
categories. Given $F$ in $\Zsp\B$ the composite $F\comp f$ is
injective (since the left adjoint of restriction along $f$ is exact)
but need not be indecomposable. Thus $f$ induces a continuous map
$\Zsp\B\to\Zsp\A$ provided that indecomposability is preserved. For
instance a quotient functor $\A\to\A/\C$ with respect to a Serre
subcategory $\C\subseteq\A$ has this property; it induces a
homeomorphism \[\Zsp\A/\C
\stackrel{\sim}\longrightarrow\{F\in\Zsp\A\mid F(\C)=0\}.\]

\subsection*{The Krull--Gabriel filtration}

Following \cite[Chap.~IV]{Ga1962} and \cite[\S6]{GJ1981} we define a
filtration of $\A$ recursively as follows:
\begin{itemize}
\item[--] $\A_{-1}$ is the full subcategory containing only the zero
  object.
\item[--] $\A_{\a}$ is the full subcategory of objects of finite
length in $\A/\A_\b$, if $\a=\b+1$.
\item[--] $\A_{\a}=\bigcup_{\g<\a}\A_\g$, if $\a$ is a limit ordinal.
\end{itemize}
If $\A=\bigcup_{\a}\A_\a$ then the smallest ordinal $\a$ such that
$\A=\A_\a$ is called \emph{Krull--Gabriel dimension} and denoted
$\KGdim\A$. In that case we say that  $\KGdim\A$ \emph{exists}.

For each ordinal $\a$, let $\Zsp_\a\A$ denote the set of functors
$F\in\Zsp\A$ such that $F(\A_\a)=0$ and $F(X)\neq 0$ for some object
$X$ which is simple in $\A/\A_\a$. This yields a bijection between the
isomorphism classes of simple objects in $\A/\A_\a$ and the elements
in $\Zsp_\a\A$.

\begin{prop}\label{pr:ZspKGdim}
Suppose that $\KGdim\A=\a$. Then 
$\Zsp\A$ equals the disjoint union $\bigcup_{\b<\a}\Zsp_\b\A$.
\end{prop}
\begin{proof} 
See \cite[Theorem~12.7]{Kr1998b}.
\end{proof}

Removing successively from $\Zsp\A$ the points in $\Zsp_\b\A$ for
$\b=-1,0,1,\ldots$ yields the \emph{Cantor-Bendixson filtration} of
$\Zsp\A$, provided that $\KGdim\A$ exists. This follows from the next
lemma.

\begin{lem}
  Let $F\in\Zsp\A$. If $F(X)\neq 0$ for some finite length object $X$
  then $F$ is \emph{isolated}, that is, $\{F\}$ is open. The converse
  holds when $\KGdim\A$ exists.
\end{lem}
\begin{proof}
  If $F(X)\neq 0$ for some finite length object $X$ then we may assume
  that $X$ is simple. Thus $\{F\}=(X)$, since $F$ is an injective
  envelope of $\Hom(-,X)$ in $\Lex(\A^\op,\Ab)$. For the converse, see
  \cite[Lemma~12.11]{Kr1998b}.
\end{proof}

\begin{lem}\label{le:KGdim}
  Let $(f_i\colon\A\to\A_i)_{i\in I}$ be a family of quotient functors
  and set $\U_i=\{F\in\Zsp\A\mid F \text{ factors through }f_i\}$ for
  each $i$.  Suppose that $\Zsp\A=\bigcup_i\U_i$ and that each $\U_i$
  is an open subset. Then $\KGdim\A=\sup_i\KGdim\A_i$.
\end{lem}
\begin{proof}
  The assumption on each $\U_i$ to be open implies that
  $f_i(\A_\a)=(\A_i)_\a$ for all $i$ and each ordinal $\a$. On the
  other hand, $(\A_i)_\a=\A_i$ for all $i$ implies $\A_\a=\A$, since
  $\Zsp\A=\bigcup_i\U_i$. From this the assertion follows.
\end{proof}

\subsection*{Triangulated categories}

Let $G$ be a group of automorphisms acting on $\A$. Then we denote by
$\Zsp\A/G$ the corresponding orbit space of $\Zsp\A$. Thus the points
in $\Zsp\A/G$ are the equivalence classes of the form
$[F]=\{F\comp\gamma\mid\gamma\in G\}$. The closed subsets  correspond
to Serre subcategories of $\A$ that are $G$-invariant.

Let $\C$ be an essentially small triangulated category with suspension
$\Si\colon\C\xto{\sim}\C$. We identify cohomological functors
$\C^\op\to\Ab$ with exact functors $(\Ab\C)^\op\to\Ab$ via
Lemma~\ref{le:Freyd} and denote by $\Zsp\C$ the orbit space
$(\Zsp\Ab\C)/\Si$ with respect to the action of $\Si$.

\begin{lem}\label{le:compact}
The space $\Zsp\C$ is quasi-compact iff $\C$ admits a generator,
that is, an object not contained in any proper thick subcategory of $\C$.
\end{lem}
\begin{proof}
Suppose first that $\C$ has  a generator, say $X$. Then any $\Si$-invariant
Serre subcategory of $\Ab\C$ containing $\Hom(-,X)$ equals
$\Ab\C$. Thus  $\Zsp\C$ is quasi-compact by Lemma~\ref{le:Zsp}. To show the converse,
consider for each $X\in\C$ the closed subset \[\U_X=\{[F]\in\Zsp\C\mid
F(Y)=0\text{ for all }Y\in\Thick(X)\}.\]
Then $\bigcap_{X\in\C}\U_X=\varnothing$. If  $\Zsp\C$ is
quasi-compact, then there are finitely many objects such that
$\U_{X_1}\cap\ldots \cap\U_{X_r}=\varnothing$. This implies $\C=\Thick(X_1,\ldots,X_r)$.
\end{proof}

\end{appendix}

\end{document}